\documentclass[reqno,12pt,letterpaper]{amsart}

\RequirePackage{amsmath,amssymb,amsthm,graphicx,mathrsfs,url}
\RequirePackage[usenames,dvipsnames]{color}
\RequirePackage[colorlinks=true,linkcolor=Red,citecolor=Green]{hyperref}
\RequirePackage{amsxtra}
\usepackage{comment}

\relax

\numberwithin{equation}{section}

\def\arXiv#1{\href{http://arxiv.org/abs/#1}{arXiv:#1}}

\newtheorem{theo}{Theorem}
\newtheorem{prop}{Proposition}[section]
\newtheorem{defi}[prop]{Definition}
\newtheorem{lemm}[prop]{Lemma}

\DeclareMathOperator{\Ell}{ell}

\let\Im=\Imag 
\DeclareMathOperator{\loc}{loc}
\DeclareMathOperator{\Op}{Op}

\let\Re=\Real 

\DeclareMathOperator{\supp}{supp}

\DeclareMathOperator{\WF}{WF}

\newcommand{\RR}{{\mathbb R}}

\title[Sharp radial estimates]%
{Sharp radial estimates in Besov spaces}
\author{Jian Wang}
\email{wangjian@berkeley.edu}
\address{Department of Mathematics, University of California, Berkeley,
CA 94720}

\begin{document}

\begin{abstract}
We prove sharp radial estimates using Besov spaces. We also prove the propagation of singularities in Besov spaces.
\end{abstract}

\maketitle

\section{Introduction}

\subsection{Main results}
Radial estimates are powerful tools in microlocal analysis. They were first introduced by Melrose \cite{mel} in the study of scattering theory for asymptotically Euclidean manifolds. They have been developed in various settings since then and part of the literature is listed here. Hassell-Melrose-Vasy \cite{hmv} used radial estimates to study the scattering theory for zeroth order symbolic potentials. Vasy \cite{va} and Datchev-Dyatlov \cite{dd} used radial estimates in the study of asymptotically hyperbolic scattering theory. Vasy \cite{va}, Dyatlov \cite{dy} and Hintz-Vasy \cite{hv} applied radial estimates to general relativity. In hyperbolic dynamics, radial estimates were introduced by Dyatlov-Zworski \cite{zeta}. Radial estimates were also applied to forced waves by Dyatlov-Zworski \cite{wave} and Colin de Verdi\`ere \cite{co}.

We first review the radial estimates presented in \cite[\S E.4]{res}. We will work in mocrolocal setting and put $h=1$ in \cite[\S E.4]{res} correspondingly. Let $M$ be a smooth manifold and $P\in \Psi^{k}(M)$, $k>0$, be a properly supported pseudodifferential operator. We define
\begin{equation}
    \Re{P}:=\tfrac{P+P^*}{2}, \quad \Im{P}:=\tfrac{P-P^*}{2i}.
\end{equation}
Then $\Re{P}, \Im{P}\in \Psi^{k}(M)$ are self-adjoint and $P=\Re{P}+i\Im{P}$. 
\begin{enumerate}
\item
(\cite[Theorem E.52]{res})
Suppose $p, q\in S^{k}(T^*M;\RR)$ are the principal symbols of $\Re{P}$ and $-\Im{P}$. We consider the rescaled Hamiltonian flow $\varphi_t:=\exp(t\langle \xi \rangle H_p)$ on the compactified cotangent bundle $\overline{T}^*M$, see \S \ref{preliminaries}. Suppose $\Lambda^{-}$ is the radial source of $\varphi_t$ (for the precise meaning, see \S \ref{preliminaries}). If $q=0$ near $\Lambda^-$ and 
\begin{equation}
    \langle \xi \rangle^{1-k}\left( \sigma_{k-1}(\Im P)+(s+\tfrac{1-k}{2})\tfrac{H_p\langle \xi \rangle}{\langle \xi \rangle} \right)
\end{equation}
is eventually negative on $\Lambda^-$ with respect to $p$, then for any $B_1\in \Psi^0(M)$ such that $\Lambda^-\subset \Ell(B_1)$, there exist $A\in \Psi^0(M)$, $\chi\in C_c^{\infty}(M)$, $N\in\RR$, such that $\Lambda^-\subset \Ell(A)$, $\WF(A)\subset \Ell(B_1)$ and for any $u\in H_{\loc}^{s}$, $Pu\in H^{s-k+1}_{\loc}$, we have
\begin{equation}
\label{sobsource}
    \|Au\|_{H^s}\leq C\|B_1Pu\|_{H^{s-k+1}}+C\|\chi u\|_{H^{-N}}.
\end{equation}
Note that since $\Lambda^-$ is the radial source, there exists a maximal $s_-\in \RR$ such that \eqref{sourcecri} holds near $\Lambda^-$ for any $s>s_-$.
\item
(\cite[Theorem E.54]{res})
On the other hand, suppose $\Lambda^+$ is a radial sink of $\varphi_t$ (see \S \ref{preliminaries}), $q=0$ near $\Lambda^+$ and \eqref{sourcecri} is eventually negative on $\Lambda^+$ with respect to $p$. Then for any $B_1\in \Psi^0(M)$ such that $\Lambda^+\subset \Ell(B_1)$, there exists $A, B\in \Psi^0(M)$ such that $\Lambda^+\subset \Ell(A)$, $\WF(B)\subset \Ell(B_1)\setminus \Lambda^+$, and there exists $\chi\in C_c^{\infty}(M)$ such that for all $N$ and $u\in H^{s}_{\loc}$, $Pu\in H^{s-k+1}_{\loc}$, we have
\begin{equation}
\label{sobsink}
    \|Au\|_{H^s}\leq C\| B_1Pu \|_{H^{s-k+1}}+C\|Bu\|_{H^s}+C\|\chi u\|_{H^{-N}}.
\end{equation}
Note that when $\Lambda^+$ is the radial sink, there exists a minimal $s_+\in \RR$ such that \eqref{sourcecri} holds near $\Lambda^+$ for any $s<s_+$.
\end{enumerate}

The radial estimates \eqref{sobsource} and \eqref{sobsink} can only be applied to Sobolev spaces with restrictions on the Sobolev regularity. It is natural to ask if we can get similar estimates in other function spaces with critical regularity. In particular, if we consider Besov spaces, we have the following 

\begin{theo}
\label{theo1}
Suppose $P\in \Psi^k(M)$, $k>0$, and $\Lambda^-$ is the radial source with respect to $p=\Re{\sigma(P)}$. Assume that $\langle \xi \rangle^{-k}\Im\sigma(P)=0$ near $\Lambda^-$ and there exists $T>0$ such that
\begin{equation}
\label{sourcecri}
    \int_0^T \langle \xi \rangle^{1-k}\left(\sigma_{k-1}(\Im P)+(s_-+\tfrac{1-k}{2})\tfrac{H_p\langle \xi \rangle}{\langle \xi \rangle}\right)\circ \varphi_t dt\leq 0
\end{equation}
in a neighborhood of $\Lambda^-$.
Then there exists a conic neighborhood $U$ of $\Lambda^-$, such that for any $B_1\in \Psi^0(M)$ with $\Lambda^-\subset \Ell(B_1)\subset U$, there exists $A\in \Psi^0(M)$, $\chi\in C_c^{\infty}(M)$ such that $\Lambda^-\subset \Ell(A)$ and for any $u\in \mathscr{D}^{\prime}(M)$, if $B_1u\in B^{s_-}_{2,\infty}(M)$, $B_1Pu\in B^{s}_{2,1}(M)$ for some $s>s_-$, then for any $N\in \RR$, we have $Au\in B_{2,\infty}^{s}$ and
\begin{equation}
\label{sourceest}
    \|Au\|_{B_{2,\infty}^{s}}\leq C\|B_1 Pu\|_{B_{2,1}^{s-k+1}}+C\|\chi u\|_{H^{-N}}.
\end{equation}
\end{theo}

\begin{theo}
\label{theo2}
Suppose $P\in \Psi^k(M)$, $k>0$, and $\Lambda^+$ is the radial sink with respect to $p=\Re\sigma(P)$. Assume that $\langle \xi \rangle^{1-k}\Im\sigma(P)=0$ near $\Lambda^+$ and there exists $T>0$ such that
\begin{equation}
\label{sinkcri}
    \int_0^T\langle \xi \rangle^{1-k}\left( \sigma_{k-1}(\Im P)+\left( s_++\tfrac{1-k}{2} \right)\tfrac{H_p\langle \xi \rangle}{\langle \xi \rangle} \right) \circ \varphi_t dt \leq 0
\end{equation}
near $\Lambda^+$. Then there exists a conic neighborhood $U$ of $\Lambda^+$ such that for any $B_1\in \Psi^0(M)$ with $\Lambda^+\subset \Ell(B_1)$ and $\WF(B_1)\subset U$, there exist $A, B \in \Psi^0(M)$ such that $\Lambda^+\subset \Ell(A)$, $\WF(B)\subset \Ell(B_1)\setminus \Lambda^+$ and there exists $\chi\in C_c^{\infty}(M)$ such that for any $N\in \RR$ and $u\in \mathscr{D}^{\prime}(M)$, if $Bu\in B_{2,1}^{s_+}$, $B_1Pu\in B_{2,1}^{s_+-k+1}$, then $Au\in B_{2,\infty}^{s_+}$ and
\begin{equation}
\label{sharpsink}
    \|Au\|_{B_{2,\infty}^{s_+}}\leq C\|Bu\|_{B_{2,1}^{s_+}}+ C\|B_1Pu\|_{B_{2,1}^{s_+}}+C\|\chi u\|_{H^{-N}}.
\end{equation}
\end{theo}

Besides radial points, we can also consider the principal type of propagation. More precisely, we have

\begin{theo}
\label{theo3}
Assume $P\in \Psi^{k}(M)$ is a properly supported operator with $\sigma(P)=p-iq$, $p,q\in S^k(T^*M;\RR)$. Let $A, B, B_1 \in \Psi^0(M)$ be compactly supported and $\langle \xi \rangle^{-k}q\geq 0$ on $\WF(B_1)$. If the following control condition holds:
for any $(x,\xi)\in \WF(A)$, there exists $T>0$ such that
\begin{equation}
\varphi_{-T}(x,\xi)\in \Ell(B), \quad \text{and} \quad \varphi_t(x,\xi)\in \Ell(B_1),\forall t\in [-T,0].
\end{equation}
Then there exists $\chi\in C^{\infty}_c(M)$ such that for any $s, N\in \RR$, if $u\in \mathscr{D}^{\prime}(M)$, $Bu\in B^s_{2,1}$, $B_1Pu\in B_{2,1}^{s-k+1}$, then $Au\in B_{2,1}^s$ and
\begin{equation}
\label{gbprop}
\|A u\|_{B^{s}_{2,1}}\leq C\|B u\|_{B^{s}_{2,1}}+C\|B_1 P u\|_{B^{s-k+1}_{2,1}}+C\|\chi u\|_{H^{-N}}
\end{equation}
Similarly, if $u\in \mathscr{D}^{\prime}(M)$, $Bu\in B_{2,\infty}^s$, $B_1Pu\in B_{2,\infty}^{s-k+1}$, then $Au\in B_{2,\infty}^s$ and
\begin{equation}
\label{gdbprop}
\|A u\|_{B^{s}_{2,\infty}}\leq C\|B u\|_{B^s_{2,\infty}}+C\|B_1 P u\|_{B^{s-k+1}_{2,\infty}}+C\|\chi u\|_{H^{-N}}.
\end{equation}
\end{theo}

\subsection{Organization of the paper.}
In \S \ref{preliminaries}, we review some basis notions and tools including Besov spaces, regularizing operators and radial sets.  In \S \ref{ssourpf}, we prove the sharp source estimates, Theorem \ref{theo1}. In \S \ref{ssinkpf}, we prove the sharp sink estimates, Theorem \ref{theo2}. In \S \ref{pspf}, we prove the propagation of singularities in Besov spaces.

\medskip\noindent\textbf{Acknowledgements.}
I would like to thank Maciej Zworski for suggesting this problem and for helpful discussion. I would like to thank Peter Hintz for sending me an unpublished note, in which he discussed the sharp radial estimates heuristically. I would also like to thank Thibault de Poyferr\'e for his interest in this note and helpful discussion. Partial support by the National Science Foundation grant DMS-1500852 is also gratefully acknowledged.

\section{Preliminaries}

\label{preliminaries}

\subsection{Besov spaces}
Let $M$ be a smooth manifold. We fix a metric $|\cdot|$ on the cotangent bundle $T^*M$. Let $\alpha, \alpha_0\in C^{\infty}(T^*M)$ such that $\alpha_0(x,\xi)=1$ when $|\xi|\leq 1$, $\alpha_0(x,\xi)=0$ when $|\xi|>2$, $\alpha(x,\xi)=1$ when $C_1\leq |\xi|\leq C_2$ for some $C_1, C_2>0$ and $\alpha=0$ when $|\xi|\leq C_1/2$ or $|\xi|>2C_2$, and
\begin{equation}
    0<C_1\leq \alpha_0(x,\xi)+\sum_{j=0}^{\infty}\alpha(x,h_0^{j}\xi)\leq C_2<\infty
\end{equation}
with $0<h_0<1$. A function $u\in \mathscr{D}^{\prime}(M)$ is in the Besov space $B_{2,1}$ if the quantity
\begin{equation}
    \|u\|_{B^s_{2,1}}:= \|\alpha_0(x,D)u\|_{L^2}+ \sum_{j=0}^{\infty}h_0^{-sj}\|\alpha(x,h_0^{j}D)u\|_{L^2(M)}
\end{equation}
is finite. $\|\cdot\|_{B^s_{2,1}}$ is the $B_{2,1}^{s}$ norm of $u$. We can also define the space $B_{2,\infty}^s$ by putting
\begin{equation}
    \|u\|_{B_{2,\infty}^s}:=\sup_{j\geq 0}\{ \|\alpha_0(x,D)u\|_{L^2}, h_0^{-sj}\|\alpha(x,h_0^{j}D)u\|_{L^2} \}.
\end{equation}
A function $u$ is in $B_{2,\infty}^s$ if and only if the norm $\|u\|_{B_{2,\infty}^s}$ is finite.
The space $B_{2,\infty}^s$ is the dual space of $B_{2,1}^{-s}$. 

We record a mapping property of pseudodifferential operators between Besov spaces.

\begin{lemm}
Suppose $P\in \Psi^k(M)$, $k\in \RR$. Then 
\begin{equation}
    P: B_{2,\infty}^s(M) \to B_{2,\infty}^{s-k}(M), \quad P: B_{2,1}^s(M) \to B_{2,1}^{s-k}(M)
\end{equation}
are bounded with norms depend on the seminorms of the principal symbol of $P$ in $S^{k}_{1,0}(T^*M)$.
\end{lemm}

\begin{proof}
Without loss of generality, we assume $k=0$. Suppose $h_0\in (0,1)$, $\alpha\in C^{\infty}(T^*M)$, $\alpha(x,\xi)=1$ when $C_1\leq |\xi|\leq C_2$ and $\alpha(x,\xi)=0$ when $|\xi|\leq C_1/2$ or $|\xi|\geq 2C_2$ with $C_1, C_2>0$. Let $\widetilde{\alpha}\in C^{\infty}(T^*M)$ such that $\widetilde{\alpha}(x,\xi)=1$ when $C_1/2\leq |\xi|\leq 2C_2$ and $\widetilde{\alpha}(x,\xi)=0$ when $|\xi|\leq C_1/4$ or $|\xi|\geq 4C_2$. Then
\begin{equation}
\label{mapping}
\begin{split}
    \|\alpha(x,h_0^jD)Pu\|_{L^2}
    = & \| P \alpha(x,h_0^jD)u \|_{L^2}+\|[\alpha(x,h_0^jD),P]\widetilde{\alpha}(x,h_0^jD)u\|_{L^2}\\
    \leq & \|P\|_{L^2\to L^2}\|\alpha(x,h_0^jD)u\|_{L^2}\\
    & +\|[\alpha(x,h_0^jD),P]\|_{H^{-1}\to L^2}\|\widetilde{\alpha}(x,h_0^jD)u\|_{H^{-1}}.
\end{split}
\end{equation}
For any $h>0$, we have in local coordinates that
\begin{equation}
    \sigma([P, \alpha(x,hD)])=H_p\left( \alpha(x,h\xi)\right)=h(H_p\alpha)(x,h\xi).
\end{equation}
Therefore
\begin{equation}\begin{split}
    & \langle \xi \rangle^{|\beta|+1}\left| \partial_{x}^{\alpha}\partial_{\xi}^{\beta}\sigma\left([ P, \alpha(x,hD) ]\right) \right|\\
    = & h^{|\beta|+1}\langle \xi \rangle^{|\beta|+1}\left| \left(\partial_x^{\alpha}\partial_{\xi}^{\beta}H_p\alpha\right)(x,h\xi) \right|\\
    \leq & C\left| \left(\partial_x^{\alpha}\partial_{\xi}^{\beta}H_p\alpha\right)(x,h\xi) \right|\leq C\sup_{T^*M}\left| \left(\partial_x^{\alpha}\partial_{\xi}^{\beta}\{p,\alpha\}\right)\right|
\end{split}\end{equation}
This shows that 
\begin{equation}
    \|[\alpha(x,h_0^jD), P]\|_{H^{-1}\to L^2}
\end{equation}
is bounded uniformly by seminorms of $P$ since $\alpha=0$ when $|\xi|\leq C_1/2$ or $|\xi|>2C_2$. Now \eqref{mapping} shows that 
\begin{equation}
    \|Pu\|_{B_{2,\infty}^{s}}\leq C\|u\|_{B_{2,\infty}^s}, \quad \|Pu\|_{B_{2,1}^s}\leq C\|u\|_{B_{2,1}^s},
\end{equation}
with $C>0$ depends on the seminorms of $\sigma(P)$.
\end{proof}

\subsection{Regularizing operators.}
In order to regularize distributions, we define for any $0<\tau\leq 1$, $r\in \RR$,
\begin{equation}
\label{xtaudefi}
    X_{\tau}:=\Op\left(\langle \tau \xi \rangle^{-m}\right)\in \Psi^{-m}(M).
\end{equation}
We record some useful properties of $X_{\tau}$:
\begin{lemm}
\label{xylemm}
Let $X_{\tau}$ be as in \eqref{xtaudefi}. Then
\begin{enumerate}
    \item There exists $Y_{\tau}\in \Psi^{m}(M)$ such that
    \begin{equation}
        Y_{\tau}=\Op(y_{\tau}), \quad y_{\tau}(x,\xi)=y(x,\tau\xi;\tau), \quad y_{\tau}=\langle \xi \rangle^m+O(\tau)_{S^{m-1}_{1,0}},
    \end{equation}
    and
    \begin{equation}
        X_{\tau}Y_{\tau}=I+O(\tau^{\infty})_{\Psi^{-\infty}}, \quad Y_{\tau}X_{\tau}=I+O(\tau^{\infty})_{\Psi^{-\infty}}.
    \end{equation}
    \item For any $P\in \Psi^{k}(M)$, we have
    \begin{equation}
        X_{\tau}PY_{\tau}=P+i\Op\left(\langle \tau\xi \rangle^{m}\{\sigma(P), \langle \tau\xi \rangle^{-m}\}\right)+O(1)_{\Psi^{m-2}_{0}(M)}.
    \end{equation}
    Moreover, we have
    \begin{equation}
        \WF(X_{\tau}PY_{\tau})=\WF(P).
    \end{equation}
    \item If $u\in \mathscr{D}^{\prime}(M)$ and for $s\in \RR$, 
    \begin{equation}
        \|X_{\tau}u\|_{B_{2,\infty}^s}\leq C
    \end{equation}
    for some finite constant $C$ and any $\tau\in(0,1]$, then
    \begin{equation}
        u\in B_{2,\infty}^s \quad \text{and}\quad \|u\|_{B_{2,\infty}^s}\leq C.
    \end{equation}
    The same conclusion holds for $B_{2,1}^{s}$ norms.
\end{enumerate}
\end{lemm}

\subsection{Radial sets}
Let $P\in \Psi^k(M)$ be a properly supported pseudodifferential operator with principal symbol $\sigma(P)=p(x,\xi)-iq$, $p,q\in S^k(T^*M;\RR)$, where $S^k(T^*M)$ is the polyhomogeneous symbol class of order $k$ (see \cite[Definition E.2]{res}). To define the radial set with respect to $p$, we consider the radial compactification of $T^*M$. We identify the compactified cotangent bundle $\overline{T}^*M$ with $B^*M:=\{(x,\xi): |\xi|\leq 1\}$ by the diffeomorphism $\kappa: T^*M\rightarrow B^*M$, $(x,\xi)\mapsto (x,\xi/\langle \xi \rangle)$, $\langle \xi \rangle=\sqrt{1+|\xi|^2}$. The boundary of the cotangent bundle, $\partial T^*M$, is then identified with the unit sphere bundle $S^*M$. The rescaled Hamiltonian flow $\langle \xi \rangle^{1-k}H_p$ extends to a smooth vector field on $\overline{T}^*M$ and we use $\varphi_t:=e^{t\langle \xi \rangle^{1-k}H_p}: \overline{T}^*M\rightarrow \overline{T}^*M$ to denote the flow generated by $\langle \xi \rangle^{1-k}H_p$.

\begin{defi}
\label{radialdefi}
Suppose $p\in S^k(T^*M;\RR)$. 
A set $\Lambda^{-}\subset \{(x,\xi): \langle \xi \rangle^{-k} p(x,\xi)=0\}\subset\partial \overline{T}^*M$ is a called a radial source with respect to $p$ if 
\begin{enumerate}
    \item $\Lambda^-$ is invariant under the flow $\varphi_t$;
    \item There exists an open conic neighborhood $U\subset \overline{T}^*M$ of $\Lambda^-$ such that for any $(x,\xi)\in U$, $\varphi_t(x,\xi)\rightarrow \Lambda^-$ as $t\rightarrow -\infty$;
    \item There exist $C,\theta>0$ such that for any $t<0$, $(x,\xi)\in U$, $|\varphi_t(x,\xi)|\geq C e^{t\theta}|\xi|$.
\end{enumerate}
A set $\Lambda^+$ is called a radial sink of $P$ if it is a radial source of $-p$.
\end{defi}

Now we record some useful functions.

\begin{lemm}
\label{sourcefun}
Suppose $k>0$, $p\in S^k(T^*M;\RR)$ is homogeneous of order $k$. Suppose $\Lambda^-$ is the radial source with respect to $p$ as in Definition \ref{radialdefi}. Then there exist $\chi_1, f_1\in C^{\infty}(T^*M\setminus 0;\RR_{\geq 0})$ such that 
\begin{enumerate}
    \item $\chi_1$ is homogeneous of order $0$, $0\leq \chi_1\leq 1$, $\supp\chi_1\subset U$, $\chi_1=1$ on $\Lambda^-$, and $\langle \xi \rangle^{1-k}H_p\chi_1\leq 0$;
    \item $f_1$ is homogeneous of order $1$, $f_1\geq C|\xi|$, and $|\xi|^{1-k}H_pf_1\leq -Cf_1$ on $U$.
\end{enumerate}
\end{lemm}
The proof presented here is a modification of the proof of \cite[Lemma C.1]{zeta}.
\begin{proof}
We first construct $\chi_1$. Let $\iota : T^*M\setminus 0 \to S^*M=\partial \overline{T}^*M, (x,\xi)\mapsto (x,\xi/|\xi|)$ be the natural projection. Since $|\xi|^{1-k}H_p$ is homogeneous of order $0$, we know $V:=\iota_*(|\xi|^{1-k}H_p)$ is a well-defined vector field on $S^*M$. We only need to construct $\vartheta\in C^{\infty}(S^*M)$ such that $0\leq \vartheta \leq 1$, $\supp \vartheta\subset \iota(U)$, $\vartheta=1$ on $\Lambda^-$ and $V\vartheta\leq 0$, and then put $\chi_1:=\iota^*(\vartheta)$. 

Let $\vartheta_0\in C^{\infty}(S^*M;[0,1])$ such that $\supp \vartheta_0\subset \iota(U)$, $\vartheta_0=1$ near $\Lambda^-$. By the definition of the radial source, there exists $T>0$ such that if $\vartheta_0(x,\xi)\neq 0$, then $\vartheta_0\left(e^{-tV}(x,\xi)\right)=1$ for any $t>T$. Now we put
\begin{equation}
    \vartheta(x,\xi):=T^{-1}\int_{T}^{2T}\vartheta_0(e^{tV}(x,\xi))dt.
\end{equation}
If $\vartheta(x,\xi)\neq 0$, then there exists $t_0\geq T$ such that $\vartheta_0(e^{t_0V}(x,\xi))\neq 0$, by the assumption on $T$, we find $\vartheta_0(x,\xi)=1$, which implies $(x,\xi)\in \iota(U)$. Hence $\supp \vartheta\subset \iota(U)$. Since $\Lambda^-$ is invariant under $e^{tV}$, we know $\vartheta|_{\Lambda^-}=1$. Note that
\begin{equation}
V\vartheta(x,\xi)=\tfrac{1}{T}\left( \vartheta_0\left(e^{2T}(x,\xi)\right)-\vartheta_0\left(e^{T}(x,\xi)\right) \right).
\end{equation}
Hence if $V\vartheta(x,\xi)>0$, then
\begin{equation}
    0\leq \vartheta_0\left( e^{T}(x,\xi) \right)<\vartheta_0\left( e^{2T}(x,\xi) \right)\leq 1.
\end{equation}
This is impossible since $\vartheta_0\left( e^{2T}(x,\xi) \right)\neq 0$ implies $\vartheta_0\left( e^{T}(x,\xi) \right)=1$.

We now construct $f_1$. For that we fix a smmooth metric $|\cdot|$ on $T^*M$. Let $V_1:=|\xi|^{1-k}H_p$. By the definition of the radial source, there exists $T_1>0$ such that for any $(x,\xi)\in U$, $t\geq T_1$, we have $|e^{-tV_1}(x,\xi)|\geq 2|\xi|$. We now define
\begin{equation}
    f_1(x,\xi):=\int_0^{T_1} \left| e^{-t H_p}(x,\xi) \right| dt.
\end{equation}
Since $V$ is homogeneous of order $0$, $|\xi|$ is homogeneous of order $1$, we know $f_1$ is homogeneous of order $1$. Note that $f_1>0$, thus by the homogeneity, $f_1(x,\xi)\geq C|\xi|$ for some $C>0$ and any $(x,\xi)\in T^*M$. Finally for any $(x,\xi)\in U$, we have
\begin{equation}
    V_1f_1(x,\xi)=|\xi|-\left| e^{-T_1H_p}(x,\xi) \right|\leq -|\xi|\leq -Cf_1.
\end{equation}
This concludes the proof.
\end{proof}
We can construct similar functions for the radial sink:
\begin{lemm}
\label{sinkfun}
Suppose $k>0$, $p\in S^{k}(T^*M;\RR)$ is homogeneous of order $k$. Suppose $\Lambda^+$ is the radial sink with respect to $p$ as in Definition \ref{radialdefi}. Then there exist $\chi_2, f_2\in C^{\infty}(T^*M\setminus 0; \RR_{\geq 0})$ such that
\begin{enumerate}
    \item $\chi_2$ is homogeneous of order $0$, $0\leq \chi_2\leq 1$, $\supp \chi_2\subset U$, $\chi_2=1$ on $\Lambda^+$, and $\langle \xi \rangle^{1-k}H_p\chi_2\geq 0$;
    \item $f_2$ is homogeneous of order $1$, $f_2\geq C|\xi|$, and $|\xi|^{1-k}H_pf_2\geq Cf_2$ on $U$.
\end{enumerate}
\end{lemm}

\section{Proof of sharp source estimates}

\label{ssourpf}

We now prove Theorem \ref{theo1}, the sharp source estimates. The proof here is a modification of the proof of \cite[Theorem E.52]{res}.

\begin{proof}[Proof of Theorem \ref{theo1}]
\textbf{Step 1. A priori estimates.}
We first prove that if $P$, $s_-$, $A$, $B_1$ satisfy conditions in Theorem \ref{theo1}, then for any $u\in \mathscr{D}^{\prime}(M)$, if $Au\in B_{2,\infty}^{s_-}$, $B_1Pu\in B_{2,1}^{s_--k+1}$, then 
\begin{equation}
\label{apriorisource}
    \|Au\|_{B_{2,\infty}^{s_-}}\leq C\|B_1 P u\|_{B_{2,1}^{s_--k+1}}+C\|\chi u\|_{H^{-N}}.
\end{equation}

We write $P=P_0+iQ$ with $P_0=\Im{P}$, $Q=\Im{P}$. Let $\chi_1$, $f_1$ be as in Lemma \ref{sourcefun}. Let $\rho_1\in C^{\infty}_c((1/4,\infty))$ such that $\rho_1=0$ on $(0,1/2)$, $\supp\rho_1^{\prime}\subset [1/2,4]$, $\rho_1^{\prime}=1$ on $[1,2]$. By \eqref{sourcecri} and \cite[Proposition E.51]{res}, there exists $b\in S^0(T^*M;\RR)$ such that
\begin{equation}
    \langle \xi \rangle^{1-k}\left( \sigma_{k-1}(Q)+\left(s_-+\tfrac{1-k}{2}\right)\tfrac{H_p\langle \xi \rangle}{\langle \xi \rangle}+H_pb \right)\leq 0
\end{equation}
near $\Lambda^-$. For $h>0$, we put
\begin{equation}\begin{split}
    g_{h}:=\langle \xi \rangle^{s_-+\frac{1-k}{2}}\chi_1\rho_1(h f_1)e^b, \quad 
    G_{h}:=\Op(g_{h})\in \Psi^{s_-+\frac{1-k}{2}}(M).
\end{split}\end{equation} 
Now we have 
\begin{equation}
    G^*_hG_h u\in B_{2,\infty}^{-s_-+k-1}(M), \quad Pu\in B^{s-k+1}_{2,1},
\end{equation}
hence the pairing $\langle Pu, G_h^*G_h u \rangle$ is finite and
\begin{equation}
\begin{split}
    \Im\left\langle Pu, G_{h}^*G_{h}u \right\rangle =\langle F_{h}u,u \rangle
\end{split}
\end{equation}
with
\begin{equation}
    F_{h}:=\tfrac{i}{2}[P_0, G_{h}^*G_{h}]+\tfrac{G_{h}^*G_{h}Q+QG_{h}^*G_{h}}{2}\in \Psi^{2s_-}(M).
\end{equation}
Note that
\begin{equation}
\begin{split}
    \sigma(F_{h,\epsilon})
    = & \langle \xi \rangle^{2s_-+1-k}\chi_1^2\rho_1^2(hf_1)e^{2b}\\
    & \times \left( \sigma_{k-1}(Q)+\left( s_-+\tfrac{1-k}{2} \right)\tfrac{H_p\langle \xi \rangle}{\langle \xi \rangle}+H_p b \right)\\
    & + \langle \xi \rangle^{2s_-+1-k}e^{2b}\left(\chi_1H_p\chi_1 \rho_1^2(hf_1) +h\chi_1^2(\rho_1\rho_1^{\prime})(hf_1)H_pf_1 \right) \\
    \leq & -C\langle \xi \rangle^{2s_-}\chi_1^2(\rho_1\rho_1^{\prime})(hf_1).
\end{split}
\end{equation}
Let
\begin{equation}
    e_h:=\chi_1\sqrt{(\rho_1\rho_1^{\prime})(hf_1)}, \quad E_h:=\Op(e_h).
\end{equation}
We can choose $\rho_1$ such that $\sqrt{\rho_1\rho_1^{\prime}}\in C_c^{\infty}(\RR)$. Then we find
\begin{equation}
    \langle \xi \rangle^{-2s_-}\sigma\left(-F_{h}-Ch^{-2s_-} E_h^*E_h\right)\geq 0.
\end{equation}
Let $\widetilde{\chi}_1\in C^{\infty}(T^*M)$ such that $\supp\widetilde{\chi}_1\subset \Ell(B_1)$ and $\widetilde{\chi}_1=1$ on $\supp \chi_1$. Let $\widetilde{\rho}_1\in C_c^{\infty}(\RR_+)$ such that $\widetilde{\rho}_1=1$ on $\supp \sqrt{\rho_1\rho_1^{\prime}}$. We put
\begin{equation}
    \widetilde{E}_h:=\Op(\widetilde{\chi}_1\widetilde{\rho}_1(hf_1))\in \Psi^0(M), \quad A:=\Op(\chi_1).
\end{equation}
Let $B_2\in \Psi^0(M)$ such that
\begin{equation}
    \supp \widetilde{\chi}_1\subset \Ell(B_2), \quad \WF(B_2)\subset \Ell(B_1).
\end{equation}
By sharp G\aa rding inequality (see for instance \cite[Proposition E.34]{res}), there exist $C>0$, $\chi\in C_c^{\infty}(M)$ such that for any $N\in \RR$, uniformly in $h$,
\begin{equation}
    \left\langle (-F_h-Ch^{-2s_-}E_h^*E_h)\widetilde{E}_hu,\widetilde{E}_hu \right\rangle\geq -C\|B_2 \widetilde{E}_h u\|^2_{H^{s_--1/2}}-C\|\chi u\|^2_{H^{-N}}.
\end{equation}
By elliptic estimates (see for instance \cite[Theorem E.33]{res}), uniformly in $h$,
\begin{equation}\begin{split}
    & \|E_hu\|_{L^2}\leq C\|E_{h}\widetilde{E}_hu\|_{L^2}+C\|\chi u\|_{H^{-N}}, \\
    & \|B_2\widetilde{E}_hu\|_{H^{s_--1/2}}\leq C\|B_2u\|_{H^{s_--1/2}}+C\|\chi u\|_{H^{-N}},\\
    & \langle F_h\widetilde{E}_hu, \widetilde{E}_hu \rangle=\langle F_hu,u \rangle+C\|\chi u\|_{H^{-N}}.
\end{split}\end{equation}
Hence we find
\begin{equation}
\begin{split}
    h^{-2s_-}\|E_h u\|_{L^2}\leq \left| \langle G_hPu, G_h u \rangle \right| + C\|B_2 u \|_{H^{s_--1/2}}+C\|\chi u\|_{H^{-N}}.
\end{split}
\end{equation}
Note that
\begin{equation}
    |\langle G_h P u, G_h u \rangle|\leq \|G_h P u\|_{B_{2,1}^{\frac{1-k}{2}}}\|G_h u\|_{B_{2,\infty}^{\frac{k-1}{2}}}\leq \|B_1 P u\|_{B_{2,1}^{s_--k+1}}\|Au\|_{B_{2,\infty}^{s_-}}.
\end{equation}
Thus we find
\begin{equation}
    h^{-2s_-}\|E_hu\|_{L^2}\leq \|B_1Pu\|_{B_{2,1}^{s_--k+1}}\|Au\|_{B_{2,\infty}^{s_-}}+C\|B_2 u\|_{H^{s_--1/2}}+C\|\chi u\|_{H^{-N}}.
\end{equation}
Take the supreme over $h>0$ and then use Cauchy-Schwarz inequality and we find
\begin{equation}
\label{beforeremove}
    \|Au\|_{B_{2,\infty}^{s_-}}\leq C\| B_1 P u \|_{B_{2,1}^{s_--k+1}}+C\| B_2u \|_{H^{s_--1/2}}+C\|\chi u\|_{H^{-N}}.
\end{equation}
To remove $\|B_2 u\|_{H^{s_--1/2}}$ from the right hand side, we use the propagation estimates (see for instance \cite[Theorem E.47]{res}): the conditions for $(A, B, B_1)$ in \cite[Theorem E.47]{res} are satisfied by $(B_2, A, B_1)$ in our case. Hence we have
\begin{equation}
\label{propremove}
    \|B_2 u\|_{H^{s_--1/2}}\leq C \|B_1Pu\|_{H^{s_--k+1/2}}+C\|Au\|_{H^{s_--1/2}}+C\|\chi u\|_{H^{-N}}.
\end{equation}
Note that for any $r\in \RR$ we have
\begin{equation}
    H^r\subset B_{2,\infty}^r\subset \cup_{r^{\prime}>0}H^{r-r^{\prime}}, \quad \cap_{r^{\prime}>0}H^{r+r^{\prime}}\subset B_{2,1}^r\subset H^r.
\end{equation}
By the interpolation inequality \cite[Proposition E.21]{res}, for any $\delta>0$, there exists $C(\delta)$ such that
\begin{equation}
    \|Au\|_{H^{s_--1/2}}\leq \delta \|Au\|_{H^{s_--3/4}}+C(\delta)\|\chi u\|_{H^{-N}}\leq C\delta\|Au\|_{B_{2,\infty}^{s_-}}+C(\delta)\|\chi u\|_{H^{-N}}.
\end{equation}
Let $\delta=1/(2C)$ and combine \eqref{beforeremove}, \eqref{propremove}, we get \eqref{apriorisource}.

\noindent\textbf{Step 2. Regularization.}
Now we can prove Theorem \ref{theo1} by using regularizing operators. More precisely, let $X_{\tau}$, $Y_{\tau}$ be as in \eqref{xtaudefi} and Lemma \ref{xylemm}. 
We consider
\begin{equation}\begin{split}
    P_{\tau}:=X_{\tau}P Y_{\tau}=& P+[X_{\tau},P]Y_{\tau}+O(\tau^{\infty})_{\Psi^{-\infty}(M)}\\
    = & P+i\Op(\langle \tau\xi \rangle^{s-s_-}\{ \sigma(P), \langle \tau\xi \rangle^{s_--s} \})+O(1)_{\Psi^{k-2}_0(M)}.
\end{split}\end{equation}
We have $\sigma(\Re P_{\tau})=p$, and near $\Lambda^-$
\begin{equation}\begin{split}
    \sigma_{k-1}(\Im P_{\tau})
    = & q+\langle \tau \xi \rangle^{s-s_-}\{p,\langle \tau \xi \rangle^{s_--s}\} \\
    = & q-\tfrac{s-s_-}{2}\tfrac{\tau^2\langle\xi\rangle^2}{\langle \tau\xi \rangle^2}\tfrac{H_p\langle \xi \rangle}{\langle \xi \rangle}.
\end{split}\end{equation}
Therefore
\begin{equation}\begin{split}
    & \langle \xi \rangle^{1-k}\left( \sigma_{k-1}(\Im P_{\tau})+\left(s+\tfrac{1-k}{2}\right)\tfrac{H_p\langle \xi \rangle}{\langle \xi \rangle} \right) \\
    = & \langle \xi \rangle^{1-k}\left( \sigma_{k-1}(P) +\left( s+\tfrac{1-k}{2}-\tfrac{s-s_-}{2}\tfrac{\tau^2\langle \xi \rangle^2}{\langle \tau\xi \rangle^2} \right)\tfrac{H_p\langle \xi \rangle}{\langle \xi \rangle} \right).
\end{split}\end{equation}
Since
\begin{equation}
    s-\tfrac{s-s_-}{2}\tfrac{\tau^2\langle \xi \rangle^2}{\langle \tau \xi \rangle^2} \geq s_-,
\end{equation}
we know that there exists $T>0$ such that near $\Lambda^-$, 
\begin{equation}
    \int_0^T\langle \xi \rangle^{1-k}\left( \sigma_{k-1}(\Im P_{\tau})+\left( s+\tfrac{1-k}{2} \right)\tfrac{H_p\langle \xi \rangle}{\langle \xi \rangle} \right)\circ\varphi_t dt\leq 0.
\end{equation}
If we put
\begin{equation}
    A_{\tau}:=X_{\tau}AY_{\tau}, \quad B_{1,\tau}:=X_{\tau}B_1Y_{\tau},
\end{equation}
then
\begin{equation}
    A_{\tau}X_{\tau}u\in B_{2,\infty}^{s}, \quad B_{1,\tau}PX_{\tau}u\in B_{2,1}^{s-k+1}.
\end{equation}
Thus we can apply \eqref{apriorisource} to $P_{\tau}$, $s$, $B_{1,\tau}$, $A_{\tau}$ and $X_{\tau}u$ and we find that uniformly in $\tau$ we have
\begin{equation}\begin{split}
    \|X_{\tau}Au\|_{B_{2,\infty}^s} & \leq C\|X_{\tau}B_1Pu\|_{B_{2,1}^{s-k+1}}+C\|\chi u\|_{H^{-N}}\\
    & \leq C\|B_1Pu\|_{B_{2,1}^{s-k+1}}+C\|\chi u\|_{H^{-N}}<\infty.
\end{split}\end{equation}
By Lemma \ref{xylemm} we know
\begin{equation}
    Au\in B_{2,\infty}^s
\end{equation}
and
\begin{equation}
    \|Au\|_{B_{2,\infty}^s}\leq C\|B_1Pu\|_{B_{2,1}^{s-k+1}}+C\|\chi u\|_{H^{-N}}.
\end{equation}
This concludes the proof.
\end{proof}

\section{Proof of sharp sink estimates}
\label{ssinkpf}

We prove Theorem \ref{theo2}, the sharp sink estimates, in this section. The proof we present here is a modification of the proof of \cite[Theorem E.54]{res}.

\begin{proof}[Proof of Theorem \ref{theo2}.]
\textbf{Step 1. A priori estimates.}
We first show that if $P$, $s_+$, $A$, $B$, $B_1$ satisfy conditions in Theorem \ref{theo2}, and $Au\in B_{2,\infty}^{s_+}$, $Bu\in B_{2,1}^{s_+}$, $B_1Pu\in B_{2,1}^{s_+-k+1}$, then \eqref{sharpsink} holds.

Let $P_0$, $Q$ be as in the proof of Theorem \ref{theo1}. Let $\chi_2$, $f_2$ be as in Lemma \ref{sinkfun}. Let $\rho_2\in C^{\infty}_c(\RR)$ such that $\rho_2=1$ near $0$ and $\rho_2^{\prime}\leq 0$ on $[0,\infty)$.
By the assumption \eqref{sinkcri} and \cite[Proposition E.51]{res}, there exists $b\in S^0(T^*M;\RR)$ such that
\begin{equation}
    \langle \xi \rangle^{1-k}\left( \sigma_{k-1}(Q)+\left( s_++\tfrac{1-k}{2} \right)\tfrac{H_p\langle \xi \rangle}{\langle \xi \rangle}+H_p b \right)\leq 0
\end{equation}
near $\Lambda^+$. Now we put
\begin{equation}
    G_h:=\Op(g_h), \quad g_h:=\langle \xi \rangle^{s_++\frac{1-k}{2}}\chi_2\rho_2(h f_2)e^b.
\end{equation}
Note that $\rho_2(h f_2)$ is a symbol in $S^0_{1,0}(T^*M)$. In fact, one can check that for any $\alpha, \beta$, there exists $C_{\alpha,\beta}>0$ that is independent of $h$ such that
\begin{equation}
    \sup_{(x,\xi)\in T^*M}\langle \xi \rangle^{|\beta|}\left| \partial_{x}^{\alpha}\partial_{\xi}^{\beta}\rho_2(h f_2) \right|\leq C_{\alpha, \beta}.
\end{equation}
Similar to the proof of Theorem \ref{theo1}, we have
\begin{equation}
    \Im\langle Pu, G_h^*G_h u \rangle_{L^2}=\langle F_h u, u \rangle_{L^2}
\end{equation}
with
\begin{equation}
    F_h:=\tfrac{i}{2}[P_0,G_h^*G_h]+\tfrac{G_h^* G_h Q + Q G^*_h G_h}{2} \in \Psi^{2s_+}(M).
\end{equation}
Here we used the assumption that $\langle \xi \rangle^{-k}\sigma_{k}(Q)=0$ on $\Lambda^+$.
Note that
\begin{equation}\begin{split}
    \langle \xi \rangle^{-2s_+}\sigma_{2s_+}(F_h)= & \langle \xi \rangle^{1-k}\left( \sigma_{k-1}(Q)+(s_++\tfrac{1-k}{2})\tfrac{H_p\langle \xi \rangle}{\langle \xi \rangle}+H_p b \right)\chi_2^2\rho^2_2(h f_2)e^{2b} \\
    & + \langle \xi \rangle^{1-k}(\chi_2 H_p\chi_2)\rho_2^2(hf_2)e^{2b}+h\langle \xi \rangle^{1-k}\chi_2^2(\rho_2\rho_2^{\prime})(hf_2)(H_p f_2)e^{2b}\\
    \leq & C\langle \xi \rangle^{1-k}(\chi_2 H_p\chi_2)-C\chi_2^2(\rho_2\rho^{\prime}_2)(h f_2).
\end{split}\end{equation}
Now we put
\begin{equation}\begin{split}
    A:= & \Op(\chi_2), \quad B:=\Op\left(\langle \xi \rangle^{\frac{1-k}{2}}\sqrt{\chi_2 H_p \chi_2}\right), \\ E_h:= & \Op\left( \sqrt{\rho_2\rho_2^{\prime}}(hf_2) \right), \quad Y:=\Op\left( \langle \xi \rangle^{s_+} \right).
\end{split}\end{equation}
Let $B_2\in \Psi^0(M)$ such that $\WF(B_2)\subset \Ell(B_1)$, $\supp\chi_2\subset \Ell(B_2)$. Then by sharp G\aa rding inequality, there exist $\chi\in C^{\infty}_c(M)$, $C_1>0$ such that for any $N$ and $u\in C^{\infty}(M)$, we have, uniformly in $h$,
\begin{equation}\begin{split}
    & \langle \left(-C(E_h Y A)^*(E_h Y A)+C(YB)^*(YB)-F_h\right)u,u \rangle \\
    \geq & -C_1\|B_2u\|^2_{H^{s_+-1/2}}-C_1\|\chi u\|^2_{H^{-N}}.
\end{split}\end{equation}
Hence we find
\begin{equation}\begin{split}
    \|E_h Y A u\|^2_{L^2}\leq & C| \langle G_h P u, G_h u \rangle |+C\|Bu\|^2_{H^{s_+}}+C\| B_2u \|^2_{H^{s_+-1/2}} + C\|\chi u\|^2_{H^{-N}}\\
    \leq & C\| B_1Pu \|_{B^{s_+-k+1}_{2,1}}\|Au\|_{B^{s_+}_{2,\infty}} \\ & +C\|Bu\|^2_{H^{s_+}}+C\|B_2u\|^2_{H^{s_+-1/2}}+C\|\chi u\|^2_{H^{-N}}.
\end{split}\end{equation}
Take the supreme for $h>0$ and use Cauchy-Schwarz inequality and we find
\begin{equation}
    \|Au\|_{B^{s_+}_{2,\infty}}\leq C\|B_1Pu\|_{B^{s_+-k+1}_{2,1}}+C\| Bu \|_{H^{s_+}}+\|B_2u\|_{H^{s_+-1/2}}+C\|\chi u\|_{H^{-N}}.
\end{equation}
The $\|B_2 u\|_{H^{s_+-1/2}}$ term can be removed as in the proof of Theorem \ref{theo1}.

\noindent\textbf{Regularization.}
We now prove Theorem \ref{theo2}. 
Since $u\in \mathscr{D}^{\prime}(M)$, there exists $N\in \RR$ such that $Au\in B^{-N}_{2,\infty}$. Let $X_{\tau}$, $Y_{\tau}$ be as in \eqref{xtaudefi} and Lemma \ref{xylemm} with $m=s_++N$. We can choose $N$ large enough such that $m>0$. Then we have
\begin{equation}\begin{split}
    & \langle \xi \rangle^{1-k} \left(\Im(P_{\tau})+\left(s_++\tfrac{1-k}{2}\right)\tfrac{H_p\langle \xi \rangle}{\langle \xi \rangle}\right) \\
    = & \langle \xi \rangle^{1-k}
    \left( \Im(P)+ \left( s_+ + \tfrac{1-k}{2}-\tfrac{s_++N}{2}\tfrac{\tau^2\langle \xi \rangle^2}{\langle \tau \xi \rangle^2} \right)\tfrac{H_p\langle \xi \rangle}{\langle \xi \rangle}\right).
\end{split}\end{equation}
Since
\begin{equation}
    s_++\tfrac{1-k}{2}-\tfrac{s_++N}{2}\tfrac{\tau^2\langle \xi \rangle^2}{\langle \tau \xi \rangle^2} \leq s_++\tfrac{1-k}{2},
\end{equation}
we know there exists $T>0$ such that
\begin{equation}
    \int_{0}^T\langle \xi \rangle^{1-k} \left(\Im(P_{\tau})+\left(s_++\tfrac{1-k}{2}\right)\tfrac{H_p\langle \xi \rangle}{\langle \xi \rangle}\right)\circ \varphi_t dt \leq 0
\end{equation}
near $\Lambda^+$. Note that
\begin{equation}\begin{split}
    & A_{\tau}X_{\tau}u\in B^{s_+}_{2,\infty}, \quad B_{\tau}X_{\tau}u \in B^{2s_++N}_{2,1}\subset B_{2,1}^{s_+}, \\
    & B_{1,\tau}P_{\tau}X_{\tau}u\in B_{2,1}^{2s_++N-k+1}\subset B_{2,1}^{s_+-k+1},
\end{split}\end{equation}
hence by the a priori estimates we have
\begin{equation}\begin{split}
    \|X_{\tau}Au\|_{B_{2,\infty}^{s_+}}
    \leq & C\|X_{\tau}Bu\|_{B_{2,1}^{s_+}}+C\|X_{\tau}B_1Pu\|_{B_{2,1}^{s_+-k+1}}+C\|\chi u\|_{H^{-N}}\\
    \leq & C\|Bu\|_{B_{2,1}^{s_+}}+C\|B_1Pu\|_{B_{2,1}^{s_+-k+1}}+C\|\chi u\|_{H^{-N}}.
\end{split}\end{equation}
Since this is true for any $\tau\in (0,1]$, we conclude that $Au\in B_{2,\infty}^{s_+}$ and \eqref{sharpsink} holds.
\end{proof}

\section{Propagation of singularities}

\label{pspf}

In this section, we modify the proof of \cite[Theorem E.47]{res} to prove the propagation of singularities in Besov spaces. We first construct some functions that are useful in the proof of Theorem \ref{theo3}.
\begin{lemm}
\label{gcons}
Suppose $A, B, B_1\in \Psi^0$ satisfy conditions in Theorem \ref{theo3}. Then for any $\beta>0$, we can construct $g, g_1\in C^{\infty}(T^*M\setminus 0;\RR_{\geq 0})$ such that
\begin{enumerate}
    \item $g$ is homogeneous of order $0$, $\supp g\subset \Ell(B_1)$, $g\geq 0$, $g>0$ on $\WF(A)$, and $|\xi|^{1-k}H_p g \leq -\beta g$ in a conic neighborhood of $(T^*M\setminus 0)\setminus \Ell(B)$;
    \item $g_1$ is homogeneous of order $1$, $g_1>0$ in a conic neighborhood of $\supp g$, and $|\xi|^{1-k}H_p g_1=0$ in a conic neigborhood of $\supp g$.
\end{enumerate}
\end{lemm}

\begin{proof}
We identify the boundary of $T^*M\setminus 0$ with the co-sphere bundle $S^*M$ and conic subsets of $T^*M\setminus 0$ with subsets of $S^*M$ by using the radial compactification $\iota: T^*M\setminus 0\to S^*M, (x,\xi)\mapsto (x,\xi/|\xi|)$.

We first assume that $\WF(A)\subset S^*M$ is a single point $\{(x_0,\xi_0)\}$ and there exists $T>0$ such that $e^{-TV}(x_0,\xi_0)\in \Ell(B)$, $e^{tV}(x_0,\xi_0)\in \Ell(B_1)$ for $t\in [-T,0]$. Here $V:=\iota_*(|\xi|^{1-k}H_p)$ which is homogeneous of order $0$ hence extends to a smooth vector field on $S^*M$.

To construct $g$, we only need to construct $\theta\in C^{\infty}_{c}(S^*M)$ such that $\supp \theta\subset \Ell(B_1)$, $\theta\geq 0$, $\theta>0$ on $\WF(A)$, and $V\theta\leq -\beta\theta$ in a neighborhood of $S^*M\setminus \Ell(B)$.
Let $\Sigma\subset S^*M$ be a hypersurface passing through $(x_0,\xi_0)$ such that the map
\begin{equation}
    \Phi: (-T-\delta,\delta)\times \Sigma\to S^*M, \quad (t,(x,\xi)) \mapsto e^{tV}(x,\xi)
\end{equation}
is a diffeomorphism onto its image. We can assume $\delta$ and $\Sigma$ is small enough such that 
\begin{equation}
    \Phi((-T-\delta, -T+\delta)\times \Sigma)\subset \Ell(B); \quad \Phi((-T-\delta,\delta)\times \Sigma)\subset \Ell(B_1).
\end{equation}
Let $\psi_0\in C^{\infty}_c((-T-\delta,\delta);\RR_{\geq 0})$ such that 
\begin{equation}
    \psi_0\geq 0, \quad \psi_0(0)>0, \quad \psi_0^{\prime}\leq -\beta\psi_0 ~~~~\text{outside}~~~~(-T-\delta/2,\delta/2). 
\end{equation}
Let $\varphi_0 \in C_c^{\infty}(\Sigma)$ such that $\varphi_0(x_0,\xi_0)>0$. Now for $(x,\xi)\in \Phi((-T-\delta,\delta)\times \Sigma)$, we define
\begin{equation}
    \theta\circ e^{tV}(x,\xi):=\psi_0(t)\varphi_0(x,\xi)
\end{equation}
and extend by zero outside. We put $g:=\iota^*\theta$.

We now construct $g_1$. Let $\varphi_1\in C^{\infty}_c(\Sigma)$ such that $\varphi_1=1$ on $\supp \varphi_0$. We define
\begin{equation}
    \theta_1\circ e^{tV}(x,\xi):=\varphi_1(x,\xi)|\xi|.
\end{equation}
We put $g_1:=\iota^*\theta_1$.

Finally, in the general case, for every $(x,\xi)\in \WF(A)$, we can construct $g_{(x,\xi)}$, $g_{1,(x,\xi)}$ in a small conic neighborhood of $(x,\xi)$. If such neighborhoods of $(x_1,\xi_1), \cdots, (x_m, \xi_m)$ form an open covering of $\WF(A)$, then we can put
\begin{equation}
    g:=\sum g_{(x_j,\xi_j)}, \quad g_1:=\sum g_{1,(x_j,\xi_j)}.
\end{equation}
One can check $g$ and $g_1$ satisfy the conditions.
\end{proof}

We can now prove Theorem \ref{theo3}.
\begin{proof}[Proof of Theorem \ref{theo3}.]
For simplicity, we only prove the a priori estimates, that is, we assume $u\in\mathscr{D}^{\prime}(M)$ satisfies $Au\in B_{2,1}^s$, $Bu\in B_{2,1}^{s}$ and $B_1Pu\in B_{2,1}^{s-k+1}$ in the proof of \eqref{gbprop}. The estimates \eqref{gbprop} then follows the same regularization argument as in the proof of Theorem \ref{theo1} and of Theorem \ref{theo2}. Similarly, in the proof of \eqref{gdbprop}, we assume $Au\in B_{2,\infty}^s$, $Bu\in B_{2,\infty}^s$ and $B_1Pu\in B_{2,\infty}^{s-k+1}$.

We now prove the a priori estimates of \eqref{gbprop}. Let $g, g_1$ be as in Lemma \ref{gcons}. Let $\rho_3\in C^{\infty}_c(\RR_{+})$ such that $\rho_3=1$ on $[1,2]$, $\rho_3=0$ on $(0,1/2)\cup (4,\infty)$. We now put
\begin{equation}
    G_h:=\Op(\langle \xi \rangle^{s+\frac{1-k}{2}}g\rho_3(hg_1))\in \Psi^{s+\frac{1-k}{2}}. 
\end{equation}
Here we regard $\rho_3(hg_1)$ as a symbol of order $0$ since its seminorms in $S^0_{1,0}(T^*M)$ are uniformly bounded in $h$.
We now compute
\begin{equation}
\label{commu}
    \Im\langle Pu, G_h^*G_h u \rangle=\langle \tfrac{i}{2}[P_0, G^*_hG_h]u,u \rangle+\langle QG_hu,G_h u \rangle+\langle \Re(G_h^*[G_h,Q])u, u \rangle.
\end{equation}

For the first term on the right hand side of \eqref{commu}: note that
\begin{equation}\begin{split}
    \sigma\left( \tfrac{i}{2}[P_0, G_h^*G_h] \right)
    = & \langle \xi \rangle^{2s}\left( \left(s+\tfrac{1-k}{2}\right) g^2 \langle \xi \rangle^{-k}H_p\langle \xi \rangle +g \langle \xi \rangle^{1-k}H_p g \right)\rho_3^2(hg_1)\\
    \leq & (C_1-\beta)\langle \xi \rangle^{2s}g^2\rho_3^2(hg_1)
\end{split}\end{equation}
near $(T^*M\setminus 0)\setminus \Ell(B)$.

Let 
\begin{equation}
    K:=\{(C_1-\beta)\langle \xi \rangle^{2s}g^2\rho_3^2(hg_1)-\sigma\left( \tfrac{i}{2}[P_0, G_h^*G_h] \right)<0\},
\end{equation}
then $K\subset \Ell(B)$. Let $\widetilde{B}\in \Psi^s(M)$ such that
\begin{equation}
    K\subset \Ell(\widetilde{B}), \quad \WF(\widetilde{B})\subset \Ell(B).
\end{equation}
Put
\begin{equation}
    W_h:=\Op\left(\langle \xi \rangle^{s}g\rho_3(hg_1)\right)\in \Psi^{s}(M),
\end{equation}
\begin{equation}\begin{split}
    Z_h:=(C_1-\beta)W_h^*W_h-\tfrac{i}{2}[P_0, G_h^*G_h]+C\left(\Op(\rho_3(hg_1))\widetilde{B}\right)^*\left( \Op(\rho_3(hg_1)\widetilde{B}) \right).
\end{split}\end{equation}
Then there exists $C>0$ such that $Z_h\in \Psi^{2s}(M)$ with seminorms of $\sigma(Z_h)$ uniformly bounded in $h$ in the symbol class $S^{2s}(T^*M)$, and we have
\begin{equation}
    \langle \xi \rangle^{2s}\sigma(Z_h)\geq 0.
\end{equation}

Let $\widetilde{g}\in C^{\infty}(T^*M)$, $\widetilde{\rho}_3\in C_c^{\infty}((1/2,4))$ such that $\supp \widetilde{g}\subset \Ell(B_1)$, $\widetilde{g}=1$ on $\supp g$, $\widetilde{\rho}_3=1$ on $\supp \rho_3$. Put
\begin{equation}\begin{split}
    \widetilde{W}_h:=\Op(\widetilde{g}\widetilde{\rho}_3(hg_1))\in \Psi^0(M), \quad B_2:=\Op(\widetilde{g})\in \Psi^0(M).
\end{split}\end{equation}
Then by sharp G\aa rding inequality, there exist $C>0$ and $\chi\in C_c^{\infty}(M)$ such that
\begin{equation}
    \langle Z_h \widetilde{W}_hu, \widetilde{W}_hu \rangle\geq -C\|B_2\widetilde{W}_hu\|^2_{H^{s-1/2}}-C\|\chi u\|^2_{H^{-N}}.
\end{equation}
That is,
\begin{equation}\begin{split}
    \langle \tfrac{i}{2}[P_0, G_h^*G_h]\widetilde{W}_h u, \widetilde{W}_h u \rangle \leq &  (C_1-\beta)\|W_h\widetilde{W}_h u\|_{L^2}^2+\|\Op(\rho_3(hg_1))\widetilde{B}\widetilde{W}_h u\|_{L^2}^2 \\
    & + \|B_2\widetilde{W}_h u\|_{H^{s-1/2}}^2+C\|\chi u\|_{H^{-N}}^2.
\end{split}\end{equation}
By elliptic estimates, this implies
\begin{equation}\begin{split}
\label{p1}
    \langle \tfrac{i}{2}[P_0, G_h^*G_h]u, u \rangle \leq & (C_1-\beta)\|W_h u\|_{L^2}^2+ \|\Op(\rho_3(hg_1))\widetilde{B}u\|_{L^2}^2 \\
    & + \|B_2\widetilde{W}_hu\|_{H^{s-1/2}}^2+C\|\chi u\|_{H^{-N}}^2
\end{split}\end{equation}

For the second term on the right hand side of \eqref{commu}: since $\langle \xi \rangle^{-k}\sigma(\Im P)=-\langle \xi \rangle^{-k}q\leq 0$ on $\WF(B_1)$, again by sharp G\aa rding inequality, we have
\begin{equation}\begin{split}
\label{p2}
    \langle (\Im P)G_hu,G_hu \rangle
    \leq & C\|B_2G_hu\|^2_{H^{\frac{k-1}{2}}}+C\|\chi u\|_{H^{-N}}^2 \\
    \leq & C_2\|W_h u\|_{H^s}^2+C\| \chi u \|_{H^{-N}}^2.
\end{split}\end{equation}

For the last term on the right hand side of \eqref{commu}: note that the principal symbol of $G^*_h[G_h,P]$ is purely imaginary, we have $\Re \left( G_h^*[G_h,Q] \right)\in \Psi^{2s-1}(M)$ with operator norms bounded uniformly in $h$. Hence by the elliptic estimate, \cite[Theorem E.33]{res}, we have
\begin{equation}
    \langle \Re\left( G_h^*[G_h,Q] \right)\widetilde{W}_hu,\widetilde{W}_hu \rangle\leq C\|B_2\widetilde{W}_hu\|_{H^{s-1/2}}^2+C\|\chi u\|_{H^{-N}}^2.
\end{equation}
This implies 
\begin{equation}
\label{p3}
    \langle \Re(G_h^*[G_h,Q])u, u \rangle\leq C\|B_2 \widetilde{W}_h u\|_{H^{s-1/2}}^2+C\|\chi u \|_{H^{-N}}^2.
\end{equation}
Combine \eqref{p1}, \eqref{p2}, \eqref{p3} and we get
\begin{equation}\begin{split}
    \Im\langle Pu, G_h^*G_h u \rangle\leq & (C_1+C_2-\beta)\|W_h u\|_{L^2}^2 + \|\Op(\rho_3(hg_1))\widetilde{B}u\|_{L^2} \\
    & + C\|B_2 \widetilde{W}_hu\|_{H^{s-1/2}}^2+C\|\chi u\|_{H^{-N}}^2.
\end{split}\end{equation}
Now we choose $\beta=C_1+C_2+1$ and we find
\begin{equation}\begin{split}
    \|W_h u\|_{L^2}^2 \leq & \|\Op(\rho_3(hg_1))\widetilde{B}u\|^2_{L^2}+|\langle Pu, G_h^*G_h u \rangle| \\
    & + C\|B_2 \widetilde{W}_h u\|_{H^{s-1/2}}^2+C\|\chi u\|_{H^{-N}}^2\\
    \leq & \|\Op(\rho_3(hg_1))\widetilde{B}u\|^2_{L^2}+\|G_h Pu\|_{H^{\frac{1-k}{2}}}\|G_h u\|_{H^{\frac{k-1}{2}}}\\
    & + C\|B_2 \widetilde{W}_h u\|_{H^{s-1/2}}^2+C\|\chi u\|_{H^{-N}}^2.
\end{split}\end{equation}
By Cauchy-Schwarz inequality and notice that
\begin{equation}
    \|G_u u\|_{H^{\frac{k-1}{2}}}\leq C\| W_h u \|_{L^2}, \quad \|G_h Pu\|_{H^{\frac{1-k}{2}}}\leq C\|W_h Pu\|_{H^{1-k}},
\end{equation}
we find
\begin{equation}\begin{split}
    \|W_h u\|_{L^2} \leq &  \|\Op(\rho_3(hg_1))\widetilde{B}u\|_{L^2}+\|W_h Pu\|_{H^{1-k}} \\
    & +C\|B_2 \widetilde{W}_h u\|_{H^{s-1/2}}+C\|\chi u\|_{H^{-N}}.
\end{split}\end{equation}
Therefore we have
\begin{equation}
    \|Au\|_{B_{2,\infty}^s}\leq C\|Bu\|_{B_{2,\infty}^s}+C\|B_1 Pu\|_{B_{2,\infty}^{s-k+1}}+C\|B_2 u\|_{B_{2,\infty}^{s-1/2}}+C\|\chi u\|_{H^{-N}},
\end{equation}
and
\begin{equation}
    \|Au\|_{B_{2,1}^s}\leq C\|Bu\|_{B_{2,1}^s}+C\|B_1 Pu\|_{B_{2,1}^{s-k+1}}+C\|B_2 u\|_{B_{2,1}^{s-1/2}}+C\|\chi u\|_{H^{-N}}.
\end{equation}
This concludes the proof.
\end{proof}


\begin{thebibliography}{0}

\bibitem[Co]{co} Yves Colin de Verdi\`ere,
    \emph{ Spectral theory of pseudo-differential operators of degree 0 and application to forced linear waves,\/ }
    to appear in Analysis $\&$ PDE, \arXiv{1804.03367}.

\bibitem[DD]{dd} Kiril Datchev and Semyon Dyatlov,
    \emph{ Fractal Weyl laws for asymptotically hyperbolic manifolds, \/ }
    Geom. Funct. Anal. \textbf{23}(2013), 1145-1206.
    
\bibitem[Dy]{dy} Semyon Dyatlov,
    \emph{ Asymptotic distribution of quasi-normal modes for Kerr-de Sitter black holes, \/ }
    Ann. Henri Poincar\'e (A), \textbf{17}(2016), 3089-3146.

\bibitem[DZ1]{zeta} Semyon Dyatlov and Maciej Zworski,
    \emph{ Dynamical zeta functions for Anosov flows via microlocal analysis,\/ }
    Ann. Sci. Ec. Norm. Sup\'er. \textbf{49}(2016), 543-577.
    
\bibitem[DZ2]{wave} Semyon Dyatlov and Maciej Zworski,
    \emph{ Microlocal analysis of forced waves, \/ }
    Pure and Applied Analysis, \textbf{1}(2019), 359-394.

\bibitem[DZ3]{res} Semyon Dyatlov and Maciej Zworski,
	\emph{Mathematical theory of scattering resonances,\/}
	Graduate Study of Mathematics \textbf{200}, AMS 2019.

\bibitem[Me]{mel} Richard B. Melrose,
    \emph{ Spectral and scattering theory for the Laplacian on asymptotically Euclidean spaces, \/ }
    (M. Ikawa, ed.), Marcel Dekker, 1994.
    
\bibitem[HMV]{hmv} Andrew Hassell, Richard B. Melrose, and Andr\'as Vasy,
    \emph{ Spectral and scattering theory for symbolic potentials of order zero, \/ }
    Adv. Math. \textbf{181}(2004), 1-87.
    
\bibitem[HV]{hv} Peter Hintz and Andr\'as Vasy,
    \emph{ The global non-linear stability of the Kerr-de Sitter family of black holes, \/}
    Acta Mathematica, \textbf{220}(2018), 1-206.
    
\bibitem[Va]{va} Andr\'as Vasy,
    \emph{ Microlocal analysis of asymptotically hyperbolic and Kerr-de Sitter spaces, \/ } with an Appendix by Semyon Dyatlov,
    Invent. Math. \textbf{194}(2013), 381-513.


\end{thebibliography}
\end{document}